\pdfoutput=1
\documentclass{amsart}
\usepackage{amsfonts}
\usepackage{amsfonts}
\usepackage{mathrsfs}
\usepackage{graphicx}
\usepackage{amsfonts}
\usepackage[dvipsnames,usenames]{color}
\usepackage{amsmath, amsthm, amsfonts, amssymb}
\usepackage[toc,page]{appendix}
\usepackage{comment}

\newtheorem{thm}{Theorem}[section]

\newtheorem{lem}[thm]{Lemma}
\newtheorem{prop}[thm]{Proposition}

\newtheorem{rem}[thm]{Remark}

\theoremstyle{definition}
\numberwithin{equation}{section}
\theoremstyle{remark} \hsize=7.5truein \vsize=8.6truein

\def\C{{\hbox{\bf C}}}

\def\P{{\hbox{\bf P}}}
\def\E{{\hbox{\bf E}}}

\def\be{{\bf{e}}}
 at 10 true pt

\def\be#1{ \begin{equation}\label{#1} }
\def\bas{\begin{align*}}
\def\eas{\end{align*}}
\def\bi{\begin{itemize}}
\def\ei{\end{itemize}}

\def\to{\rightarrow}

\def\emph#1{{\it #1}}
\def\textbf#1{{\bf #1}}

\def\to{{\tilde o}}

\theoremstyle{plain}
  \newtheorem{theorem}[subsection]{Theorem}

\theoremstyle{remark}
  \newtheorem{remark}[subsection]{\bf Remark}

\theoremstyle{definition}
  \newtheorem{definition}[subsection]{Definition}

\begin{document}
\title[Random covariance matrices]{Random covariance matrices: Universality of local statistics of  eigenvalues up to the edge}
\author{Ke Wang}
\address{Department of Mathematics, Rutgers University, Piscataway, NJ 08854, US}
\email{wkelucky@math.rutgers.edu}

\begin{abstract}
We study the universality of the eigenvalue statistics of the covariance matrices $\frac{1}{n}M^* M$ where $M$ is a large $p\times n$ matrix obeying condition $\bf{C1}$. In particular, as an application, we prove a variant of universality results regarding the smallest singular value of $M_{p,n}$. This paper is an extension of the results in \cite{tvcovariance} from the bulk of the spectrum up to the edge. 
\end{abstract}

\maketitle

\section{Introduction}

The goal of this paper is to extend the Four Moment theorem established by Tao and Vu \cite{tvcovariance}  for iid covariance matrices from the bulk of spectrum to the edge. Let us first specify the matrix ensembles that will be studied.

\begin{definition}[Condition \bf{C1}]Consider a random $p\times n$ matrix $M_{n,p} =(\zeta_{ij})_{1\le i\le p, 1\le j \le n}$, where $p=p(n)$ is an integer parameter such that $p\le n$ and $\lim_{n \rightarrow \infty} p/n=y$ for some $0< y\le 1$.  We say that the matrix ensemble $M$ obeys condition $\bf{C1}$ if the random variables $\zeta_{ij}$ are jointly independent, have mean zero and variance $1$, and obey the moment condition $\sup_{i,j} \textbf{E} |\zeta_{ij}|^{C_0} \le C$ for some constant $C$ independent of $n,p$. 
\end{definition}

Given such a $p\times n$ random matrix $M$, we form the $n \times n$ \emph{covariance}  \emph{matrix} $ W= W_{n,p}= \frac{1}{n} M^*M$. This (non-negative) matrix has rank $p$ and the first $n-p$ eigenvalues are $0$. We order its remaining eigenvalues as $$0 \le {\lambda}_1 {(W)} \le {\lambda}_2 {(W)} \le \ldots \le {\lambda}_p {(W)}.$$ 


Denote $\sigma_1(M),\ldots, \sigma_p(M)$ to be the singular values of $M$. Notice that $\sigma_i(M) = \sqrt{n} \lambda_i (W)^{1/2}$. From the singular value decomposition, there exist orthonormal bases $u_1,\ldots,u_p \in \mathbb{C}^n$ and $v_1,\ldots,v_p$ $\in \mathbb{C}^p$ such that 
$$Mu_i=\sigma_i v_i$$ and $$M^*v_i=\sigma_i u_i.$$

The empirical spectral distribution (ESD) of the matrix $W$(which is Hermitian and thus has real eigenvalues) is a one-dimensional function $$F^{\bf W}(x)=\frac{1}{p} |\{ 1\le j \le p: \lambda_j(W) \le x\}|,$$

where we use $|\mathbf{I}|$ to denote the cardinality of a set $\mathbf{I}$.

The first fundamental result concerning the asymptotic limiting behavior of ESD for large covariance matrices is the  $\mathit{Marchenko-Pastur}$ $\mathit{Law}$ due to \cite{MP} (see also \cite{BS1}).

\begin{theorem}\label{thm:MP}
(Marchenko-Pastur  Law)
Assume a $p \times n$ random matrix $M$ obeys condition $\bf{C1}$ with $C_0 \ge 4$, and $p,n \rightarrow \infty$ such that $\lim_{n \rightarrow \infty} p/n=y \in (0,1]$, the empirical spectral distribution of the matrix $W_{n,p}= \frac{1}{n} M^*M$ converges in distribution to the Marchenko-Pastur  Law with a density function  $${{\rho}}_{MP,y}(x) := \frac{1}{2 \pi xy} \sqrt{(b-x)(x-a)} \mathbf{1}_{[a,b]}(x),$$ where $$a:=(1-\sqrt{y})^2, b:=(1+\sqrt{y})^2.$$
\end{theorem}

We introduce the notation of frequent events, depending on $n$, in increasing order of likelihood. 
\begin{definition}[Frequent events]\label{freq-def}\cite{tvrandom}  Let $E$ be an event depending on $n$.
\begin{itemize}
\item $E$ holds \emph{asymptotically almost surely} if $\P(E) = 1-o(1)$.
\item $E$ holds \emph{with high probability} if $\P(E) \geq 1-O(n^{-c})$ for some constant $c>0$ (independent of $n$).
\item $E$ holds \emph{with overwhelming probability} if $\P(E) \geq 1-O_C(n^{-C})$ for \emph{every} constant $C>0$ (or equivalently, that $\P(E) \geq 1 - \exp(-\omega(\log n))$).
\item $E$ holds \emph{almost surely} if $\P(E)=1$.  
\end{itemize}
\end{definition}

\begin{definition}[Matching]  We say that two complex random variables $\zeta, \zeta'$ \emph{match to order $k$} for some integer $k \geq 1$ if one has $\E \text{Re}(\zeta)^m \text{Im}(\zeta)^l = \E \text{Re}(\zeta')^m \text{Im}(\zeta')^l$ for all $m,l \geq 0$ with $m+l \leq k$. \end{definition}

Our main result is the following Four Moment theorem, which extends the result (Theorem 6) in \cite{tvcovariance} to the edge of the spectrum. The proof is analogous to the proofs in \cite{tvrandom}, \cite{tvrandom2} and \cite{tvcovariance} and will be presented in Section \ref{section:main}.
\begin{theorem}[Four Moment Theorem]\label{thm:4main}
For sufficiently small $c_0 >0$ and sufficiently large $C_0 >0$ ($C_0=10^4$ will suffice) the following holds for every $k\ge 1$. Let $M=(\zeta_{ij})_{1\le i\le p, 1\le j\le n }$ and $M'=(\zeta'_{ij})_{1\le i\le p, 1\le j\le n }$ be two random matrices satisfying condition \textbf{C1} with the indicated constant $C_0$, and assume that for each $i,j$ that $\zeta_{ij}$ and $\zeta'_{ij}$ match to order 4. Let $W, W'$ be the associated covariance matrices. Assume also that $p/n\rightarrow y$ for some $0<y \le 1$.

Let $G:\mathbb{R}^k \rightarrow \mathbb{R}$ be a smooth function obeying the derivative bounds 
\begin{equation}\label{eq:derivative}
|\nabla^j G(x)| \le n^{c_0}
\end{equation}
for all $0\le j \le 5$ and $x\in \mathbb{R}^k$.

Then for any $1\le i_1 < i_2<\ldots <i_k \le n$, and for $n$ sufficiently large depending on $k, c_0$, we have 
\begin{equation}\label{eq:conclusion}
|\E(G(n\lambda_{i_1}(W),\ldots,n\lambda_{i_k}(W) ) ) - \E(G(n\lambda_{i_1}(W'),\ldots,n\lambda_{i_k}(W') ) )| \le n^{-c_0}.
\end{equation}

If $\zeta_{ij}$ and $\zeta'_{ij}$ only match to order 3 rather 4, then the conclusion (\ref{eq:conclusion}) still holds provided that one strengthens (\ref{eq:derivative}) to 
$$|\nabla^j G(x)| \le n^{-jc_1}$$ for all $0\le j \le 5$ and $x\in \mathbb{R}^k$ and any $c_1 >0$, provided that $c_0$ is sufficiently small depending on $c_1$.
\end{theorem}

The next theorem is an extension of Theorem 17 in \cite{tvcovariance}, which is used in the proof of Theorem \ref{thm:4main} and is of independent interest as well. The proof is delayed to Section \ref{section:main}.

\begin{definition}[Gap property up to the edge] Let $M$ be a random matrix obeying condition \textbf{C1}. We say M obeys the \emph{gap property} if for every $c>0$ and every $1\le i \le p$, one has $|\lambda_{i+1}(W)-\lambda_i(W)| \ge n^{-1-c}$ with high probability.
\end{definition}

\begin{theorem}[Gap theorem up to the edge]\label{thm:gap-c1}Let $M$ be a random matrix satisfying condition \textbf{C1}. Then $M$ obeys the gap property.
\end{theorem}

\begin{remark} When $y=1$, the singular value statistics around $a=0$ turn out to be different since the density function $\rho_{\text{MP},y}(x)$ has a singularity at $x=0$. The hard edge is not really an edge, which makes it easier to deal with.  In this paper, we will focus on the edge case when $a>0$.

\end{remark}

\begin{remark} We consider $n$ as an asymptotic parameter tending
to infinity.  We use $X \ll Y$, $Y \gg X$, $Y = \Omega(X)$, or $X =
O(Y)$ to denote the bound $X \leq CY$ for all sufficiently large $n$
and for some constant  $C$. Notations such as
 $X \ll_k Y, X= O_k(Y)$ mean that the hidden constant $C$ depend on
 another constant $k$. $X=o(Y)$ or $Y= \omega(X)$ means that
 $X/Y \rightarrow 0$ as $n \rightarrow \infty$; the rate of decay here will be allowed to depend on other parameters.  We write $X = \Theta(Y)$ for $Y \ll X \ll Y$. We view vectors $x \in \C^n$ as column vectors. The Euclidean norm of a vector $x \in \C^n$ is defined as $\|x\| := (x^* x)^{1/2}$. 
\end{remark}

This paper is organized as follows: in Section 2, we prove a variant of universality result regarding the smallest singular value as an application of the Four Moment theorem. In Section 3, we mention a few basic results from linear algebra and probability. In Section 4, we provide the proofs of two technical lemmas, which are the major content of this paper. Finally, in Section 5, we give the proofs of the Gap theorem (Theorem \ref{thm:gap-c1}) and Four Moment theorem (Theorem \ref{thm:4main}). The argument draws heavily from those in \cite{tvrandom}, \cite{tvrandom2} and \cite{tvcovariance}, thus we only focus on the changes needed to complete the proofs.

\textbf{Acknowledgments:}
The author would like to thank Van H. Vu for useful discussion and his guidance  through to the completion of this paper.

\section{Applications}

In a similar way as \cite{tvcovariance} (Section 1.3), equipped with the Four Moment theorem, we can obtain universality results for large classes of random matrices. Let us demonstrate through some examples, focusing on the results for the lower edge of the spectrum. Recall $\sigma_1(M_{p,n})$ denotes the smallest singular value of $M_{p,n}$.

We adapt the notation in \cite{tvsingular}. In this section, $M_{p,n}(\zeta)$ denotes the random $n\times p$ matrix whose entries are iid copies of a (real or complex-valued) random variable $\zeta$. We say $\zeta$ is $\mathbb{R}$-\emph{normalized} ($\mathbb{C}$-\emph{normalized}) if $\zeta$ is real-valued with $\E\zeta=0$ and $\E\zeta^2=1$  (complex-valued with $\E\zeta=0, \E\text{Re}(\zeta)^2=\E\text{Im}(\zeta)^2=1/2,$ and $\E\text{Re}(\zeta)\text{Im}(\zeta)=0$). 
A complex random variable $\zeta$ of mean $0$ and variance $1$ is  \emph{Gaussian divisible} if it has the same distribution as $(1-t)^{1/2}\zeta'+t^{1/2}\zeta''$ for some $0<t<1$, where $\zeta',\zeta''$ are independent with mean $0$ and variance one, with $\zeta''$ complex Gaussian.

For the case when $p=n$, the limiting distribution for Gaussian models was computed by Edelman \cite{Edelman}.  Recently a universality result has been established by Tao and Vu \cite{tvsingular} for the entries with bounded sufficiently high moments.

For the non-square case, where $p\le n$, the distribution of the smallest singular value of $\mathbb{R}$-\emph{normalized} ($\mathbb{C}$-\emph{normalized}) entries has been studied by Borodin and Forrester \cite{BF}. Later, Feldheim and Sodin \cite{FS} proved a universality result for certain sample covariance matrices:
\begin{thm}\label{thm:FS}
Let $M_{p,n}=(\zeta_{ij})_{1\le i\le p,1\le j\le n}$ be a random covariance matrix, where $p=p(n)\le n$ tends to infinity as $n\rightarrow \infty$ and $\limsup_{n\rightarrow \infty} p/n <1$. Let $\zeta_{ij}$ be independent for all $i,j$.

If $\zeta_{ij}$ are $\mathbb{R}$-\emph{normalized}, exponential decaying and symmetric (that is, $\zeta_{ij}$ and $-\zeta_{ij}$ have the same distribution), then 
\begin{equation}\label{eq:TW1}
\displaystyle \frac{\sigma_1{(M)^2} -(p^{1/2}-n^{1/2}) }{(p^{1/2}-n^{1/2})(p^{-1/2}-n^{-1/2})^{1/3}} \rightarrow \text{TW}_1;
\end{equation}

If $\zeta_{ij}$ are $\mathbb{C}$-\emph{normalized}, exponential decaying and symmetric, then 
\begin{equation}\label{eq:TW2}
\displaystyle \frac{\sigma_1{(M)^2} -(p^{1/2}-n^{1/2}) }{(p^{1/2}-n^{1/2})(p^{-1/2}-n^{-1/2})^{1/3}} \rightarrow \text{TW}_2,
\end{equation}
where $\text{TW}_1,\text{TW}_2$ denote the Tracy-Widom distributions.
\end{thm}

In a same way as the authors proving (\cite{tvrandom}, Theorem 9) and (\cite{tvcovariance}, Theorem 11),  one can get the following (also see Figure 1 for numerical simulations):
\begin{thm}
The conclusions of Theorem \ref{thm:FS} can be extended to the case when $p=p(n)\le n$ tends to infinity as $n\rightarrow \infty$ and $\lim_{n\rightarrow \infty} p/n =y \in (0,1]$, and when $M_{p,n}=(\zeta_{ij})_{1\le i\le p,1\le j\le n}$ obeying condition \textbf{C1} with sufficiently large constant $C_0$, and $\zeta_{ij}$ have vanishing third moment.
\end{thm}

\begin{figure} [htbp]
  \centering 
  \includegraphics[scale=0.47]{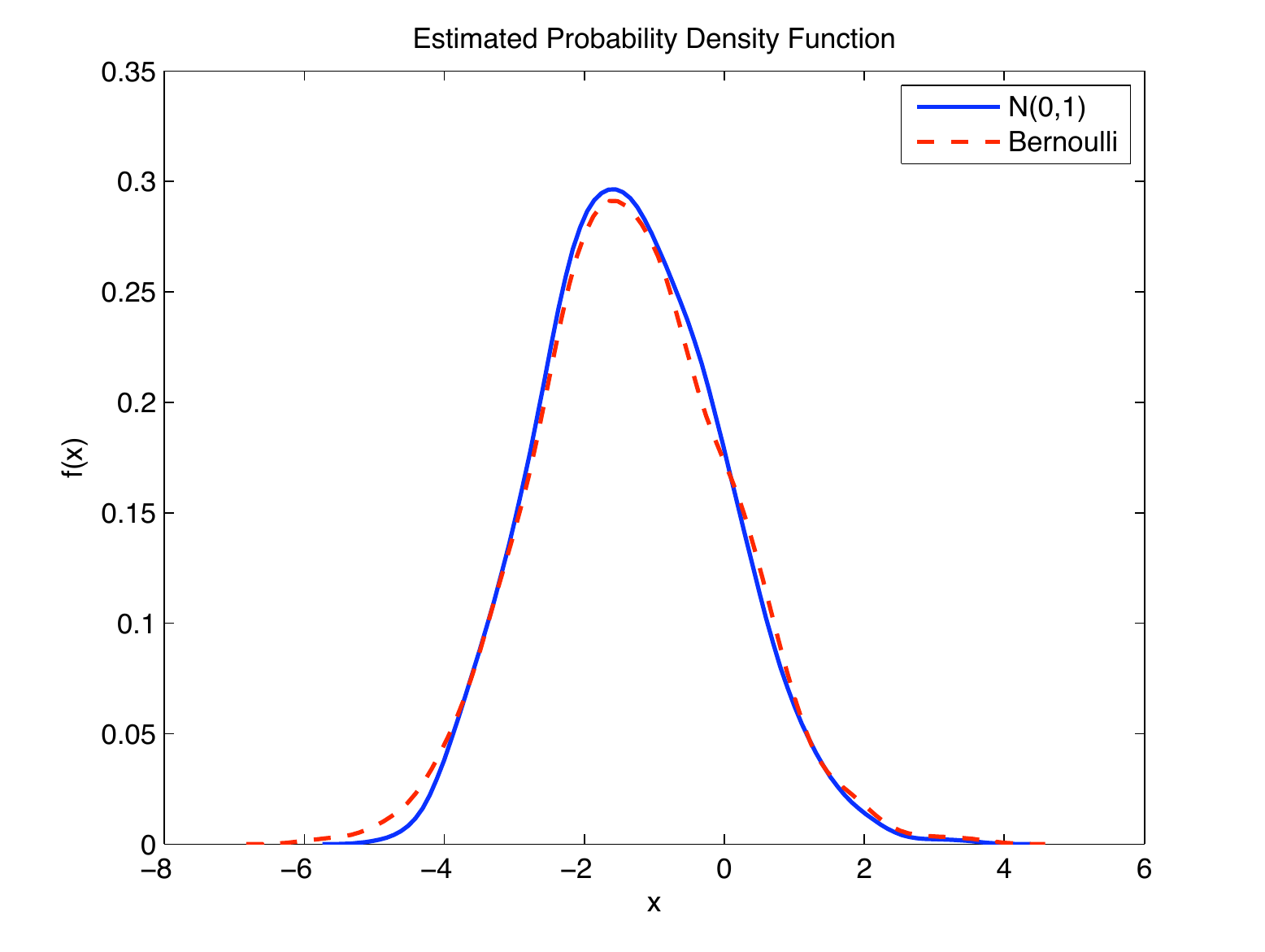}%
   \hfill{}%
   \includegraphics[scale=0.47]{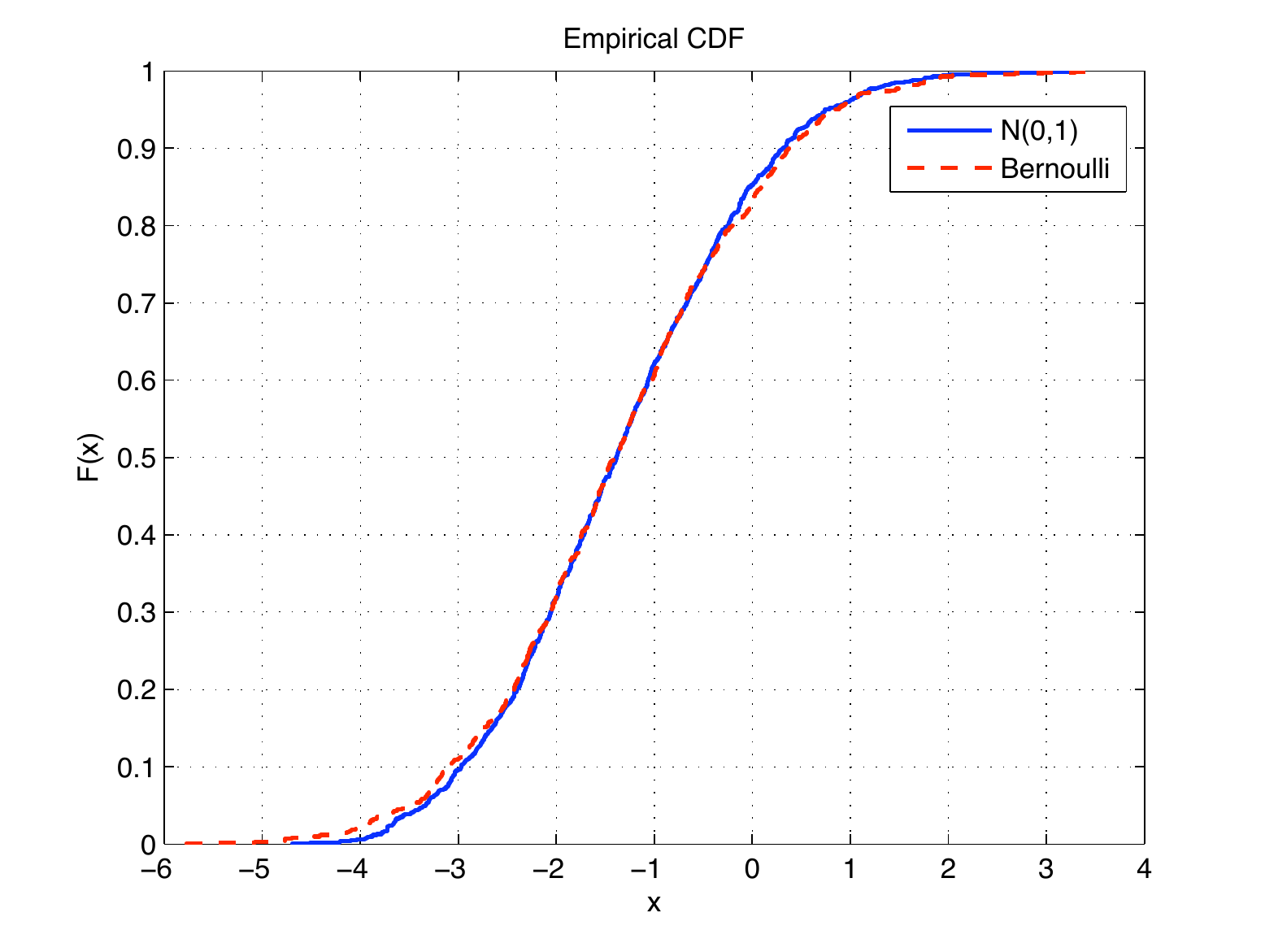} 
  \hspace{1.5in}\parbox{6in}{\caption{Plotted above are the empirical PDF and CDF of  the distribution of $\sigma_1(M_{p,n}(\zeta))^2$ (normalized as in (\ref{eq:TW1})) for $n=800,p=600$, based on data from 1000 random matrices. The blue solid curves were generated with $\zeta=N(0,1)$, while the red dashed curves were generated with $\zeta$ a random Bernoulli variable, taking values $\pm 1$ with probability $1/2$ each.}}
\end{figure}

Recently, Ben Arous and P\'ech\'e proved universality at the edge for  random matrices $M_{p,n}(\zeta)$ with i.i.d. entries of \emph{Gaussian divisible} distribution.  And with the matching theorem (Corollary 30, \cite{tvrandom}), we can drop the third moment condition whereas  $\zeta$ is assumed to be supported on at least three points.

\begin{thm}
The conclusion (\ref{eq:TW2}) of Theorem \ref{thm:FS} can be extended to the case when $p=p(n)\le n$ tends to infinity as $n\rightarrow \infty$ and $\lim_{n\rightarrow \infty} p/n=y \in (0,1] $, and when $M_{p,n}=(\zeta_{ij})_{1\le i\le p,1\le j\le n}$ obeying condition \textbf{C1} with sufficiently large constant $C_0$, and $\zeta_{ij}$ are  supported on at least three points.
\end{thm}

\section{General Tools}
In this section, we collect some basic tools from linear algebra and probability that will be used repeatedly in the sequel. 

\subsection{Tools from linear algebra}
We start with the Cauchy interlacing law and the Weyl inequalities. 

\begin{lem}[Cauchy interlacing law]\cite{tvcovariance}\label{lem:cauchy} Let $1 \leq p \leq n$.
\begin{itemize}
\item[(i)] If $A_n$ is an $n \times n$ Hermitian matrix, and $A_{n-1}$ is an $(n-1) \times (n-1)$ minor, then $\lambda_i(A_n) \leq \lambda_i(A_{n-1}) \leq \lambda_{i+1}(A_n)$ for all $1 \leq i < n$.
\item[(ii)] If $M_{p,n}$ is a $p \times n$ matrix, and $M_{p-1,n}$ is an $(p-1) \times n$ minor, then $\sigma_i(M_{p,n}) \leq \sigma_i( M_{p-1,n} ) \leq \sigma_{i+1}(M_{p,n})$ for all $1 \leq i < p$.
\item[(iii)] If $p<n$, if $M_{p,n}$ is a $p \times n$ matrix, and $M_{p,n-1}$ is a $p \times (n-1)$ minor, then $\sigma_{i-1}(M_{p,n}) \leq \sigma_i(M_{p,n-1}) \leq \sigma_i(M_{p,n})$ for all $1 \leq i \leq p$, with the understanding that $\sigma_0(M_{p,n})=0$.  (For $p=n$, one can of course use the transpose of (ii) instead.)
\end{itemize}
\end{lem}

\begin{lem}[Weyl inequality]\cite{tvcovariance}\label{lem:weyl}  Let $1 \leq p \leq n$.
\begin{itemize}
\item If $A, B$ are $n \times n$ Hermitian matrices, then $\|\lambda_i(A) - \lambda_i(B)| \leq \|A-B\|_{op}$ for all $1 \leq i \leq n$.
\item If $M, N$ are $p \times n$ matrices, then $\|\sigma_i(M) - \sigma_i(N)| \leq \|M-N\|_{op}$ for all $1 \leq i \leq p$.
\end{itemize}
\end{lem}

The following formula for an entry of a singular vector, in terms of the singular values and singular vectors of a minor, is very useful:
\begin{lem} [Corollary 25, \cite{tvcovariance}]\label{lem:coordinate}
 Let $p,n \ge 1$, and let 
$$M_{p,n}=
\left(
\begin{array}{cc}
M_{p,n-1} & X
\end{array}
\right)
 $$
be a $p\times n$ matrix for some $X\in \mathbb{C}^p$, and let 
$\left(
\begin{array}{cc}
u\\
x
\end{array}
\right)$ be a right unit singular vector of $M_{p,n}$ with singular value $\sigma_i(M_{p,n})$, where $x\in \mathbb{C}$ and $u \in \mathbb{C}^{n-1}$. Suppose that none of the singular values of $M_{p,n-1}$ are equal to $\sigma_i(M_{p,n})$. Then 
$$|x|^2= \frac{1}{1+\sum_{j=1}^{\text{min}{(p,n-1)}} \frac{\sigma_j(M_{p,n-1})^2}{(\sigma_j(M_{p,n-1})^2 - \sigma_i(M_{p,n})^2)^2} |v_j (M_{p,n-1})^* X|^2},$$
where $v_1(M_{p,n-1}), \ldots, v_{\text{min}{(p,n-1)}}(M_{p,n-1}) \in \mathbb{C}^p$ is an orthonormal system of left singular vectors corresponding to the non-trivial singular values of $M_{p,n-1}.$

In a similar vein, if $$M_{p,n}=
\left(
\begin{array}{cc}
M_{p-1,n}\\
Y^*
\end{array}
\right)$$
for some $Y\in \mathbb{C}^n$, and $\left(
\begin{array}{cc}
v\\
y
\end{array}
\right)$ is a left unit singular vector of $M_{p,n}$ with singular value $\sigma_i(M_{p,n})$, where $y\in \mathbb{C}$ and $v\in \mathbb{C}^{p-1}$, and none of the singular values of $M_{p-1,n}$ are equal to $\sigma_i(M_{p,n})$, then 
$$|y|^2= \frac{1}{1+\sum_{j=1}^{\text{min}{(p-1,n)}} \frac{\sigma_j(M_{p-1,n})^2}{(\sigma_j(M_{p-1,n})^2 - \sigma_i(M_{p,n})^2)^2} |u_j (M_{p-1,n})^* Y|^2},$$
where $u_1(M_{p-1,n}), \ldots, u_{\text{min}{(p-1,n)}}(M_{p-1,n}) \in \mathbb{C}^n$ is an orthonormal system of right singular vectors corresponding to the non-trivial singular values of $M_{p-1,n}.$
\end{lem}

The next lemma is the well-known Cauchy interlacing identities:
\begin{lem}[Lemma 40, \cite{tvrandom}] \label{lem:interlacing1}
 Let $A_n$ be a $n\times n$ Hermitian matrix, and 
$$A_n=
\left(
\begin{array}{cc}
A_{n-1} & X\\
X^* & a_{nn}
\end{array}
\right)
$$ Let $\lambda_i(A_n), 1\le i\le n$ be the eigenvalues of $A_n$ and $\lambda_j(A_{n-1}), 1\le j\le n-1$ be the eigenvalues of $A_{n-1}$. Suppose that $X$ is not orthogonal to any of the unit eigenvectors $u_j(A_{n-1})$ of $A_{n-1}$. Then we have 
$$\displaystyle\sum_{j=1}^{n-1} \frac{|u_j(A_{n-1})^* X|^2}{\lambda_j(A_{n-1})-\lambda_i(A_n)}=a_{nn}-\lambda_i(A_n)$$for every $1\le i\le n$.
\end{lem}

From this lemma, one immediately gets an interlacing identity for singular values:

\begin{lem} [Interlacing identity for singular values] \label{lem:cauchyid}
Assume the notations in Lemma \ref{lem:coordinate}, then for every $i$,
\begin{equation}\label{eq:lace1}
\displaystyle \sum_{j=1}^{{\text{min}{(p,n-1)}}} \frac{\sigma_j(M_{p,n-1})^2 |v_j (M_{p,n-1})^* X|^2}{\sigma_j(M_{p,n-1})^2 - \sigma_i(M_{p,n})^2} = ||X||^2 - \sigma_i(M_{p,n})^2,
\end{equation}
Similarly, we have 
\begin{equation}\label{eq:lace2}
\displaystyle \sum_{j=1}^{{\text{min}{(p-1,n)}}} \frac{\sigma_j(M_{p-1,n})^2 |u_j (M_{p-1,n})^* Y|^2}{\sigma_j(M_{p-1,n})^2 - \sigma_i(M_{p,n})^2} = ||Y||^2 - \sigma_i(M_{p,n})^2,
\end{equation}
\end{lem}

\begin{proof} Apply Lemma \ref{lem:interlacing1} to the matrix
\[
M_{p,n}^*M_{p,n}=
\left(
\begin{array}{cc}
M_{p,n-1}^*M_{p,n-1} & M_{p,n-1}^*X\\
X^*M_{p,n-1} & ||X||^2
\end{array}
\right)
\] with eigenvalue $\sigma_i(M_{p.n})^2$. 

Since we have $\lambda_j(M_{p,n-1}^*M_{p,n-1})=\sigma_j(M_{p,n-1})^2$ and $$u_j(M_{p,n-1}^*M_{p,n-1})^*M_{p,n-1}^*=\sigma_j(M_{p,n-1})v_j(M_{p,n-1})^*,$$ \eqref{eq:lace1} follows. Similarly, to show \eqref{eq:lace2}, apply Lemma \ref{lem:cauchyid} to the matrix
$$M_{p,n}M_{p,n}^*=
\left(
\begin{array}{cc}
M_{p-1,n}^*M_{p-1,n} & M_{p-1,n}Y\\
Y^*M_{p-1,n}^* & ||Y||^2
\end{array}
\right)
$$
\end{proof}

The \textit{Stieltjes transform} $s_n(z)$ of a Hermitian matrix $W$ is defined for $z \in \mathbb{C}$ by the formula 
$$s_n(z):=\frac{1}{n} \displaystyle\sum_{i=1}^{n} \frac{1}{\lambda_i(W)-z}.$$ By Schur's complement, it has the following alternate representation:

\begin{lem} [Lemma 39, \cite{tvrandom}]\label{StieTran}
Let $W=(\zeta_{ij})_{1\le i,j\le n}$ be a Hermitian matrix, and let $z$ be a complex number not in the spectrum of $W$. Then we have 
$$s_n(z)=\frac{1}{n} \displaystyle\sum_{k=1}^{n} \frac{1}{\zeta_{kk}-z- a^*_k (W_k -zI)^{-1} a_k }$$
where $W_k$ is the $(n-1) \times (n-1)$ matrix with the $k$th row and column of $W$ removed, and $a_k \in \mathbb{C}^{n-1}$ is the $k^\text{th}$ column of $W$ with the $k$th entry removed.
\end{lem}

\subsection{Tools from probability theory}
We will rely frequently on the next concentration of measure result for projections of random vectors.
\begin{lem} [Lemma 43,\cite{tvrandom}] \label{lem:projection}
Let $X=(\zeta_1,\ldots,\zeta_n) \in \mathbb{C}^n$ be a random vector whose entries are independent with mean zero, variance 1, and are bounded in magnitude by $K$ almost surely for some $K$, where $K \ge 10(\mathbf{E}|\xi|^4+1). $ Let $H$ be a subspace of dimension $d$ and $\pi_H$ the orthogonal projection onto H. Then 
\[
\textbf{P}(|\parallel \pi_H (X) \parallel -\sqrt{d}| \ge t) \le 10 \exp(-\frac{t^2}{10K^2}).
\]
In particular, one has 
$$\parallel \pi_H(X) \parallel= \sqrt{d}+O(K \log n)$$
with overwhelming probability.
\end{lem}

\begin{lem}[Theorem 44, \cite{tvrandom}]\label{lem:tail}  Let $1 \leq N \leq n$ be integers, and let $A = (a_{ij})_{1 \leq i \leq N; 1 \leq j \leq n}$ be an $N \times n$ complex matrix whose $N$ rows are orthonormal in $\mathbb{C}^n$, and obeying the incompressibility condition
\begin{equation}\label{summer}
 \sup_{1 \leq i \leq N; 1 \leq j \leq n} |a_{ij}| \leq \sigma
\end{equation}
for some $\sigma > 0$.  Let $\zeta_1,\ldots,\zeta_n$ be independent complex random variables with mean zero, variance one, and $\E |\zeta_{i} |^{3} \le C$ for some $C \geq 1$.  For each $1 \leq i \leq N$, let $S_i$ be the complex random variable
$$ S_i := \sum_{j=1}^n a_{ij} \zeta_j$$
and let $\vec S$ be the $\mathbb{C}^N$-valued random variable with coefficients $S_1,\ldots,S_N$.
\begin{itemize}
\item (Upper tail bound on $S_i$)  For $t \geq 1$, we have $\P( |S_i| \geq t ) \ll \exp(-ct^2) + C \sigma$ for some absolute constant $c>0$.
\item (Lower tail bound on $\vec S$)  For any $t \leq \sqrt{N}$, one has $\P( |\vec S| \leq t) \ll  O( t/\sqrt{N} )^{\lfloor N/4\rfloor} + C N^4 t^{-3} \sigma$.
\end{itemize}
The same claim holds if one of the $\zeta_i$ is assumed to have variance $c$ instead of $1$ for some absolute constant $c>0$.
\end{lem}

\section{Main technical lemmas}
Recall in the proofs of the Four Moment Theorem and the Gap Theorem as in \cite{tvcovariance}, \cite{tvrandom} and \cite{tvrandom2}, a crucial input was the \emph{Delocalization Theorem} of Erd\H{o}s, Schlein, and Yau (\cite{ESY1}, \cite{ESY2} and \cite{ESY3}). The material in this section is analogous to Section 3 of \cite{tvcovariance}. We will first extend the concentration of ESD result to the edge of the spectrum and use this concentration theorem to show the delocalization of singular vectors. The proof of the delocalization result in the edge of spectrum is significantly different from that in the bulk of spectrum as in \cite{tvcovariance}. However, similar to \cite{tvrandom2}, the Cauchy interlacing identities for singular values in Theorem \ref{lem:cauchyid} will help us to deal with this problem.

First observe that if $M=(\zeta_{ij})$ obeys \emph{condition} \textbf{C1} for some constant $C_0>0$, then by Markov's inequality and the union bound, one has $|\zeta_{ij}| \le n^{10/C_0}$ for all $i,j$ with probability $1-O(n^{-8})$. By a truncation technique (see \cite{bai2010} for details) and Lemma \ref{lem:weyl}, one may assume that $$|\zeta_{ij}| \le K:=n^{10/C_0}$$ almost surely for all $i,j$.

We will derive the eigenvalue concentration theorem (up to the edge) which is an analogue of Theorem 19 in \cite{tvcovariance}:

\begin{thm}[Concentration up to the edge] \label{thm:concentration}
Suppose that $p/n \rightarrow y$ for some $0<y\le 1$. Assume $a>0$. Let $M=(\zeta_{ij})_{1\le i\le p, 1\le j \le n}$ obey \emph{condition} $\textbf{C1}$ for some $C_0 \ge 2$ and the probability distribution of $\zeta_{ij}$ be continuous. Assume further $|\zeta_{ij}| \le K$ almost surely for some $K=o({n^{1/2} \delta^2 }{\log^{-1} n})$ for all $i,j$, where $0< \delta <1/2$ (which can depend on $n$). Then for any interval $I \subset [a,b]$ of length $|I| \ge \frac{K^2 \log^{4.5} n}{\delta^9 n}$, one has with overwhelming probability(uniformly in $I$) that the number of eigenvalues $N_I$ of $W$ in $I$ obeys the concentration estimate
\[
|N_I -p \displaystyle\int_I {{}\rho}_{MP,y}(x)\,dx| \le {\delta} p |I|.
\]
\end{thm}

As a consequence of Theorem \ref{thm:concentration}, one can deduce the following delocalization theorem:

\begin{thm}[Delocalization of singular vectors up to edge]\label{thm:delocalization}
Let the hypothesis be as in Theorem \ref{thm:concentration}, then with overwhelming probability, all the unit left and right singular vectors of $M$ have all coefficients uniformly of size at most ${K^2 {n^{-1/2}\log^{O(1)}n} }$.
\end{thm}

\begin{rem} The continuity hypothesis in the above theorems, which guarantees the singular values are almost surely simple,  is only a technical one.  In practice we are able to eliminate this hypothesis by a limiting argument using Lemma \ref{lem:weyl}.
\end{rem}

\subsection{Proof of Theorem \ref{thm:delocalization}:}


Let $u_1(M_{p,n}),\ldots,u_p(M_{p,n}) \in \mathbb{C}^n$ be the right singular vectors of $M_{p,n}$. By the union bound and symmetry, it suffices to show that $$|u_i(M_{p,n})^*e_1| \le K^2 n^{-1/2}\log^{O(1)}n $$ with overwhelming probability. The delocalization of left singular vectors can be proved similarly.

The ``bulk" case is treated in \cite{tvcovariance}.  Now we consider the edge case when $1\le i \le 0.001n$ or $0.999n\le i\le n$ (say). Using the Marchenko-Pastur law, we have with overwhelming probability that 
$$
|\lambda_i(W_{n,p})-a| \le o(1) \qquad \text{or} \qquad  |\lambda_i(W_{n,p})-b| \le o(1).
$$
By Lemma \ref{lem:coordinate}, it suffices to show with overwhelming probability that
\[
\displaystyle \sum_{j=1}^{\text{min}{(p,n-1)}} \frac{\sigma_j(M_{p,n-1})^2}{(\sigma_j(M_{p,n-1})^2 - \sigma_i(M_{p,n})^2)^2} |v_j (M_{p,n-1})^* X|^2 \gg n K^{-4}\log^{-O(1)}n.
\]
From Lemma \ref{lem:projection}, we conclude that $|v_j(M_{p,n-1})^*X| \ll K\log n$ with overwhelming probability for each $j$ (and hence for all $j$, by the union bound). Then it is enough to show that with overwhelming probability
$$\displaystyle \sum_{j=1}^{\text{min}{(p,n-1)}} \frac{\sigma_j(M_{p,n-1})^2}{(\sigma_j(M_{p,n-1})^2 - \sigma_i(M_{p,n})^2)^2} |v_j (M_{p,n-1})^* X|^4 \gg  n K^{-2}\log^{-O(1)}n.$$ 
By the Cauchy-Schwarz inequality, it thus suffices to show that 
$$\displaystyle \sum_{T_- \le j \le T_+}^{} \frac{\sigma_j(M_{p,n-1})}{|\sigma_j(M_{p,n-1})^2 - \sigma_i(M_{p,n})^2|} |v_j (M_{p,n-1})^* X|^2 \gg \sqrt n \log^{-O(1)}n$$
with overwhelming probability for some $1\le T_- < T_+ \ll K^2\log^{O(1)}n$. Noticed that $\sigma_j(M_{p,n-1})^2 = \lambda_j(M_{p,n-1}^*M_{p,n-1})=\Theta(n)$, we thus need to show 
\begin{equation}
\displaystyle \sum_{T_- \le j \le T_+}^{} \frac{\sigma_j(M_{p,n-1})^2}{|\sigma_j(M_{p,n-1})^2 - \sigma_i(M_{p,n})^2|} |v_j (M_{p,n-1})^* X|^2 \gg n \log^{-O(1)}n
\end{equation}
with overwhelming probability for some $1\le T_- < T_+ \ll K^2\log^{O(1)}n$, which is equivalent to prove that 
\begin{equation} \label{eq:01}
\displaystyle \sum_{T_- \le j \le T_+}^{} \frac{1}{n} \frac{\sigma_j(M_{p,n-1})^2 |v_j (M_{p,n-1})^* X|^2}{|\sigma_j(M_{p,n-1})^2 - \sigma_i(M_{p,n})^2|}  \gg  \log^{-O(1)}n
\end{equation}
with overwhelming probability for some $1\le T_- < T_+ \ll K^2\log^{O(1)}n$.

\medskip
In the interlacing identity in Lemma \ref{lem:cauchyid}, we have 
\begin{equation}
\displaystyle \sum_{j=1}^{{\text{min}{(p,n-1)}}} \frac{1}{n}\frac{\sigma_j(M_{p,n-1})^2 |v_j (M_{p,n-1})^* X|^2}{\sigma_j(M_{p,n-1})^2 - \sigma_i(M_{p,n})^2} = \frac{1}{n}||X||^2 -\lambda_i(W_{p,n}).
\end{equation}
By Lemma \ref{lem:projection}, one gets $\frac{1}{p}||X||^2 = 1+o(1)$ with overwhelming probability.  And since $p/n=y+o(1)$, one has 
\begin{equation} \label{eq:02}
\displaystyle \sum_{j=1}^{{\text{min}{(p,n-1)}}} \frac{1}{n}\frac{\sigma_j(M_{p,n-1})^2 |v_j (M_{p,n-1})^* X|^2}{\sigma_j(M_{p,n-1})^2 - \sigma_i(M_{p,n})^2 }=y+o(1)-\lambda_i(W_{p,n})
\end{equation}
with overwhelming probability. In order to show (\ref{eq:01}),we will evaluate
\begin{equation}
\begin{split}
&\sum_{j>i+T_+ \ \text{or} \ j< i-T_-}^{} \frac{1}{n}\frac{\sigma_j(M_{p,n-1})^2 |v_j (M_{p,n-1})^* X|^2}{\sigma_j(M_{p,n-1})^2 - \sigma_i(M_{p,n})^2}\\
&\quad =\sum_{j>i+T_+ \ \text{or} \ j< i-T_-}^{} \frac{1}{n} \frac{|v_j (M_{p,n-1})^* X|^2\lambda_j(W_{p,n-1})}{\lambda_j(W_{p,n-1}) - \lambda_i(W_{p,n})}
\end{split}
\end{equation}
for some $T_-, T_+=K^2\log^{O(1)} n$, where $W_{p,n-1}=\frac{1}{n-1} M_{p,n-1}^* M_{p,n-1}$.

The Machenko-Pastur law implies $\lambda_j(W_{p,n-1})=\Theta(1)$ for every $1\le j\le \text{min}(p,n-1)$.

Let $A>100$ be a large constant to be chosen later. From Theorem \ref{thm:concentration}, we have that (by taking $\delta=\log^{-A/20}n$) 
\begin{equation} \label{eq:03}
N_I=p\alpha_I |I| + O(|I| p \log^{-A/20}n)
\end{equation}
with overwhelming probability for any interval $I$ of length $|I|=K^2\log^A n /n$, where $\alpha_I :=\frac{1}{|I|} \int_{I} \rho_{MP,y}(x) \ dx$. For such an interval, we see from Lemma \ref{lem:projection} that with overwhelming probability
$$
\displaystyle \sum_{j: \lambda_j(W_{p,n-1})\in I}^{} \frac{1}{n} |v_j (M_{p,n-1})^* X|^2 = \frac{N_I}{n} +O({\frac{K^2\log^{A/2} n }{n}})
$$
and thus by (\ref{eq:03}) (for A large enough), 
\[
\displaystyle \sum_{j: \lambda_j(W_{p,n-1})\in I}^{} \frac{1}{n} |v_j (M_{p,n-1})^* X|^2 = y\alpha_I |I| +O(|I|\log^{-A/20}n).
\]
Set $d_I := \frac{\text{dist} (\lambda_i(W_{p,n}), I)}{|I|}$. If $d_I \ge \log n$(say), then
\begin{equation*}
 \frac{\lambda_j(W_{p,n-1})}{\lambda_j(W_{p,n-1})-\lambda_i(W_{p,n})}
=  1+ \frac{\lambda_i(W_{p,n})}{d_I |I|} +O(\frac{\lambda_i(W_{p,n})}{d_I ^2 |I|})
\end{equation*}
for all $j$ in the above sum, and since $\lambda_i(W_{p,n})=\Theta(1)$, we get
\begin{equation} \label{eq:04}
\begin{split}
\displaystyle \sum_{j: \lambda_j(W_{p,n-1})\in I}^{} \frac{1}{n} \frac{|v_j (M_{p,n-1})^* X|^2\lambda_j(W_{p,n-1})}{\lambda_j(W_{p,n-1}) - \lambda_i(W_{p,n})} 
&= y\alpha_I |I| \left(1+\frac{\lambda_i(W_{p,n})}{d_I |I|} \right)+O(\frac{\alpha_I}{d_I ^2})\\
&+ O(|I| \log^{-A/20} n)+O(\frac{\log^{-A/20}n}{d_I}) .
\end{split}
\end{equation}
We now partition the real line into intervals $I$ of length $K^2\log^{A} n /n$, and sum (\ref{eq:04}) over all intervals $I$ with $d_I \ge \log n$. Bounding $\alpha_I$ crudely by $O(1)$, we see that $\sum_{I} O(\frac{\alpha_I}{d_I ^2}) =O(\frac{1}{\log n})=o(1)$. Similarly, one has
$$
\displaystyle\sum_{I}O(|I| \log^{-A/20} n)=O(\log^{-A/20}n)=o(1)
$$
and
$$
\displaystyle\sum_{I}O(\frac{\log^{-A/20}n}{d_I}) = O(\log^{-A/20}n \log n)=o(1).
$$
 Finally, Riemann integration of the principal value integral 
$$\text{p.v.} \displaystyle\int_{a}^b \frac{\rho_{MP,y}(x)}{x-\lambda_{i}(W_{p,n})} \ dx := \lim_{\varepsilon \rightarrow 0} \int_{a \le x \le b: |x-\lambda_i(W_{p,n})|>\varepsilon} \frac{\rho_{MP,y}(x)}{x-\lambda_{i}(W_{p,n})} \ dx$$
shows that 
$$\sum_I y\alpha_I |I| \left(1+\frac{\lambda_i(W_{p,n})}{d_I |I|} \right) = \text{p.v.} \displaystyle\int_{a}^b y\frac{x\rho_{MP,y}(x)}{x-\lambda_{i}(W_{p,n})} \ dx +o(1).$$

If $|\lambda_i(W_{p,n})- a|\le o(1)$, using the formula for the Stieltjes transform, one obtains from residue calculus that 
\begin{eqnarray*}
\text{p.v.} \displaystyle\int_{a}^b y\frac{x\rho_{MP,y}(x)}{x-\lambda_{i}(W_{p,n})} \ dx 
&=& y \left(1+\text{p.v.} \lambda_i(W_{p,n})\displaystyle\int_{a}^b \frac{\rho_{MP,y}(x)}{x-\lambda_{i}(W_{p,n})} \ dx \right)\\
&=&y\left( 1+ (1-\sqrt y)^2 \frac{1}{\sqrt{y}-y}\right)+o(1)\\
&=&\sqrt y + o(1)
\end{eqnarray*}
thus 
\begin{equation}\label{eq:05}
\sum_{j>i+T_+ \ \text{or} \ j< i-T_-}^{} \frac{1}{n}\frac{\sigma_j(M_{p,n-1})^2 |v_j (M_{p,n-1})^* X|^2}{\sigma_j(M_{p,n-1})^2 - \sigma_i(M_{p,n})^2 }= \sqrt y+o(1)
\end{equation}
and in (\ref{eq:02}), 
\begin{equation}\label{eq:06}
\displaystyle \sum_{j=1}^{{\text{min}{(p,n-1)}}} \frac{1}{n}\frac{\sigma_j(M_{p,n-1})^2 |v_j (M_{p,n-1})^* X|^2}{\sigma_j(M_{p,n-1})^2 - \sigma_i(M_{p,n})^2 }=y+o(1)-\lambda_i(W_{p,n})=2\sqrt y -1+o(1)
\end{equation}
When $0<y<1$,  $\sqrt y > 2\sqrt y -1$. (\ref{eq:01}) follows by comparing (\ref{eq:05}) and (\ref{eq:06}).

\medskip
If $|\lambda_i(W_{p,n})- b|\le o(1)$, we have 
\begin{eqnarray*}
\text{p.v.} \displaystyle\int_{a}^b y\frac{x\rho_{MP,y}(x)}{x-\lambda_{i}(W_{p,n})} \ dx 
&=& y \left(1+\text{p.v.} \lambda_i(W_{p,n})\displaystyle\int_{a}^b \frac{\rho_{MP,y}(x)}{x-\lambda_{i}(W_{p,n})} \ dx \right)\\
&=&y\left( 1- (1+\sqrt y)^2 \frac{1}{\sqrt{y}+y}\right)+o(1)\\
&=&-\sqrt y +o(1)
\end{eqnarray*}
thus 
\begin{equation}\label{eq:07}
\sum_{j>i+T_+ \ \text{or} \ j< i-T_-}^{} \frac{1}{n}\frac{\sigma_j(M_{p,n-1})^2 |v_j (M_{p,n-1})^* X|^2}{\sigma_j(M_{p,n-1})^2 - \sigma_i(M_{p,n})^2 }= -\sqrt y+o(1)
\end{equation}
and in (\ref{eq:02}), 
\begin{equation}\label{eq:08}
\displaystyle \sum_{j=1}^{{\text{min}{(p,n-1)}}} \frac{1}{n}\frac{\sigma_j(M_{p,n-1})^2 |v_j (M_{p,n-1})^* X|^2}{\sigma_j(M_{p,n-1})^2 - \sigma_i(M_{p,n})^2 }=y+o(1)-\lambda_i(W_{p,n})=-2\sqrt y -1+o(1)
\end{equation}
When $0<y\le 1$,  $-\sqrt y > -2\sqrt y -1$. Then (\ref{eq:01}) follows by comparing (\ref{eq:07}) and (\ref{eq:08}).

By the concentration theorem \ref{thm:concentration} and the Cauchy interlacing law, the interval $I$ with $d_I < \log n$ will contribute at most $K^2 \log^{O(1)} n$ eigenvalues and we can set $T_-, T_+$ accordingly. The proof is now complete.

\subsection{Proof of Theorem \ref{thm:concentration}:}

We first have  a crude upper bound on the number of eigenvalues of $W$ on an interval. The proof can be found in Section 5.2, \cite{tvcovariance}.

\begin{prop} \label{prop:ubound}
(Upper bound on ESD) Assume the hypotheses in Theorem \ref{thm:concentration}, then for any interval $I \subset \mathbb{R}$ with length $|I| \ge K\log^2 n/n$, one has $$N_{I} \ll n|I|$$ with overwhelming probability, where $N_{I}$ is the number of eigenvalues in the interval $I$.
\end{prop}

The strategy is to compare the \textit{Stieltjes transform} of the ESD of matrix $W$
$$s(z):=\frac{1}{p} \displaystyle\sum_{i=1}^{p} \frac{1}{\lambda_i(W)-z},$$ 
with the \textit{Stieltjes transform} of \textit{Marchenko-Pastur~Law} 
$$s_{MP,y}(z):= \displaystyle\int_{\mathcal{R}} \frac{1}{x-z} \rho_{MP,y}(x)\,dx=\displaystyle\int_{a}^{b} \frac{1}{2\pi xy(x-z)} \sqrt{(b-x)(x-a)}\,dx.$$ 
And thanks to the next proposition, one gets control on ESD  through control on the \textit{Stieltjes transforms}.

\begin{prop}
\emph{(Lemma 29, \cite{tvcovariance})}  
\label{ESDStie}
Let $1/10 \ge \eta \ge 1/n$, and $a, b, \varepsilon, \delta >0$. Suppose that one has the bound $$|s(z)-s_{MP,y}(z)| \le \delta$$ with (uniformly) overwhelming probability for all $z$ with $ a \le \text{Re}(z) \le b$ and $\text{Im}(z) \ge \eta$. Then for any interval $I$ in $[a-\varepsilon, b+\varepsilon]$ with $|I| \ge \text{max}(2\eta, \frac{\eta}{\delta} \log \frac{1}{\delta})$, one has $$|N_I- n \displaystyle\int_I {\rho}_{sc}(x)\,dx| \le \delta n |I|$$ with overwhelming probability. 
\end{prop}

By Proposition \ref{ESDStie}, our objective is to show 
\begin{equation} \label{eq:diff}
|s(z)-s_{MP,y}(z)| =o (\delta)
\end{equation}
with (uniformly) overwhelming probability for all $z$ with $ a \le \text{Re}(z) \le b$ and $\text{Im}(z) \ge {\eta}:=\frac{K^2\log^6 n }{n \delta^8}.$

Since $s_{MP,y}(z)$ is the unique solution to the equation $$s_{MP,y}(z)+\frac{1}{y+z-1+yzs_{MP,y}(z)}=0$$ in the upper half plane (see \cite{bai2010}), we investigate a similar equation for $s(z)$.

From Lemma \ref{StieTran}, we have 
\begin{equation} \label{eq:1.6}
s(z)= \frac{1}{p} \displaystyle{\sum_{k=1}^{p} \frac{1}{\xi_{kk}-z- Y_k }}
\end{equation}
where $Y_k=a^*_k (W_{k} -zI)^{-1} a_k$, and $W_{k}$ is the matrix $W^*=\frac{1}{n}M M^*$ with the $k^{\text{th}}$ row and column removed, and $a_k$ is the $k^{\text{th}}$ row of $W$ with the $k^{\text{th}}$ element removed. Let $M_k$ be the $(p-1)\times n$ minor of $M$ with the $k^\text{th}$ row removed and $X_i^* \in \mathbb{C}^n ~(1\le i \le p)$ be the rows of $M$. Thus $\xi_{kk}={X_k}^*X_k/n= ||X_k||^2/n, a_k=\frac{1}{n}M_k X_{k}, W_k=\frac{1}{n}M_k M_k^*$. And
\begin{equation*}
Y_k  = \sum_{j=1}^{p-1} \frac{|a_k ^* v_j(M_k)|^2}{\lambda_j(W_k)-z}
= \sum_{j=1}^{p-1} \frac{1}{n} \frac{\lambda_j(W_k)|X_k^* u_j(M_k)|^2}{\lambda_j(W_k) -z}
\end{equation*}
where $u_1(M_k),\ldots,u_{p-1}(M_k) \in \mathbb{C}^{n}$ and $v_1(M_k),\ldots,v_{p-1}(M_k) \in \mathbb{C}^{p-1}$ are orthonormal right and left singular vectors of $M_k$. Here we used the facts that $a_k^* v_j(M_k)=\frac{1}{n} X_k^* M_k^*v_j(M_k)=\frac{1}{n}\sigma_j(M_k)X_k^*u_j(M_k)$ and $\sigma_j(M_k)^2=n\lambda_j(W_k)$.

\medskip
The entries of $X_k$ are independent of each other and of $W_{k}$, and have mean $0$ and variance $1$. Noticed $u_j(M_k)$ is a unit vector. By linearity of expectation we have 
$$\displaystyle\mathbf{E}(Y_k|W_{k})=\sum_{j=1}^{p-1}\frac{1}{n}\frac{\lambda_j(W_k)}{\lambda_j(W_k) -z}
=\frac{p-1}{n}+\frac{z}{n} \sum_{j=1}^{p-1}\frac{1}{\lambda_j(W_k)-z}=\frac{p-1}{n}(1+zs_k(z))$$ 
where 
$$s_{k}(z)= \frac{1}{p-1} \displaystyle{\sum_{i=1}^{p-1} \frac{1}{\lambda_i (W_{k}) -z}}$$ 
is the \textit{Stieltjes transform} for the ESD of  $W_{k}$. From the Cauchy interlacing law, we can get
$$\displaystyle{|{} s(z)- (1-\frac{1}{p}) {} s_{k}(z)|= O(\frac{1}{p} \int_{\mathbb{R}} \frac{1}{|x-z|^2}\,dx) =O(\frac{1}{p\eta})}$$ and thus 
$$\mathbf{E}(Y_k|W_{k})=\frac{p-1}{n}+z\frac{p}{n}s(z)+O(\frac{1}{n\eta})=\frac{p-1}{n}+z\frac{p}{n}s(z)+o(\delta^2).$$

In fact a similar estimate holds for $Y_k$ itself:
\begin{prop}
\label{YProp}
For $1 \le k \le n$, $Y_k= \mathbf{E}(Y_k|W_{k}) +o(\delta^2)$ holds with (uniformly) overwhelming probability for all $z$ with $a\le\text{Re}(z) \le b$ and $\text{Im}(z) \ge {\eta}$.
\end{prop}
\begin{proof}
 Decompose $$Y_k-\mathbf{E}(Y_k|W_{k})=\displaystyle\sum_{j=1}^{p-1} \frac{\lambda_j(W_k)}{n} \left( \frac{|X_k^* u_j(M_k)|^2 -1}{\lambda_j(W_k)-z} \right): =\frac{1}{n} \sum_{j=1}^{p-1} \frac{\lambda_j(W_k)}{\lambda_j(W_k)-z}R_j.$$
Let $T \subset \{1,\ldots, n-1\}$. Let $H$ be the space spanned by $\{u_j (W_{k})\}$ for $j \in T$ and $P_H$ be the orthogonal projection onto $H$. Thus $\displaystyle\sum_{j\in T} R_j =||P_H(X_k)||^2- {\text{dim}(H)}.$

By Lemma \ref{lem:projection}, we conclude with overwhelming probability 
\begin{equation} \label{eq:1.9}
|\displaystyle\sum_{j\in T} R_j|  \ll {\sqrt{|T|} K \log n+ K^2 \log^2 n}
\end{equation}
Using the triangle inequality, 
\begin{equation} \label{eq:1.10}
\sum_{j\in T} |R_j| \ll {{|T|}+ K^2 \log^2 n}.
\end{equation}
Let $z=x+\sqrt{-1} {} \eta$, where $\eta =\frac{K^2 \log^6 n }{n\delta^8}$ and $a \le x \le b$. We will use two auxiliary parameters $\alpha=\delta^2 \log^{-1.1}n$,$\delta'=\delta^2 \log^{-0.1}n$ in later estimation.

First, for those $j \in T$ such that $|\lambda_j(W_{k})-x| \le \delta' \eta $, the function $\frac{\lambda_j(W_{k})}{\lambda_j(W_{k})-x -\sqrt{-1}{} \eta}$ has magnitude $O(\frac{1}{{} \eta})$. From Proposition \ref{prop:ubound}, $|T| \ll n\delta' \eta$, the contribution for these $j \in T$,
$$\displaystyle{|\frac{1}{n}\sum_{j \in T}^{} \frac{\lambda_j(W_{k})}{\lambda_{j}(W_{k})-z} R_j|} \ll \frac{1}{n\eta} \sum_{j\in J} |R_j| \ll \frac{n\delta'\eta+K^2 \log^2 n}{n\eta}=o(\delta^2).$$
For the contribution of the remaining indices, we subdivide them as 
$$(1+\alpha)^l \delta' \eta \le |\lambda_j(W_{k})-x| \le (1+\alpha)^{l+1} \delta' \eta$$
for $0 \le l \ll \log n/\alpha$, and then sum over $l$.

For each such interval, the function $\frac{\lambda_j(W_{k})}{\lambda_j(W_{k})-x -\sqrt{-1}{} \eta}$ has magnitude $O(\frac{1}{(1+\alpha)^l \delta' \eta} )$ and fluctuates by at most $O(\frac{\alpha}{(1+\alpha)^l \delta' \eta} )$. Say $T(l)$ is the set of all $j$'s in this interval, by Proposition \ref{prop:ubound}, $|T(l)| \ll  n\alpha (1+\alpha)^l \delta' \eta$. Together with bounds (\ref{eq:1.9}), (\ref{eq:1.10}), the contribution for these $j$ on such an interval,
\begin{equation*}
\begin{split}
\displaystyle{|\frac{1}{n}\sum_{j \in T(l)}^{} \frac{\lambda_j(W_{k})}{\lambda_{j}(W_{k})-z} R_j|} &\ll \frac{1}{\alpha (1+\alpha)^l \delta' \eta} \frac{\sqrt{ \alpha (1+\alpha)^l \delta' \eta n} K \log n +K^2 \log^2 n}{n} \\
&+ \frac{\alpha}{(1+\alpha)^l \delta' \eta} \frac{\alpha (1+\alpha)^l \delta' \eta +K^2 \log^2 n}{n} \\
&\ll \frac{\sqrt{\alpha}}{ \sqrt{\delta' \eta n}} K\log n + \frac{K^2 \log^2 n}{\delta' \eta n}+\alpha^2
\end{split}
\end{equation*}
Summing over $l$ (taking into account that $l \ll \log n/\alpha$), we will get 
$$\displaystyle{|\frac{1}{n}\sum_l\sum_{j \in T(l)}^{} \frac{\lambda_j(W_{k})}{\lambda_{j}(W_{k})-z} R_j|}\ll \frac{K \log^2 n}{\sqrt{\alpha \delta' \eta n}}  + \frac{K^2 \log^3 n}{\alpha \delta' \eta n} + \alpha \log n = o(\delta^2).$$
\end{proof}

Recall $s_{MP,y}(z)$ has an explicit expression $$s_{MP,y}(z)=-\frac{y+z-1-\sqrt{(y+z-1)^2-4yz}}{2yz},$$ where we take the branch of $\sqrt{(y+z-1)^2-4yz}$ with cut at $[a,b]$ that is asymptotically $y+z-1$ as $z\rightarrow \infty$.

From (\ref{eq:1.6}) and Proposition \ref{YProp}, we have with overwhelming probability that$$s(z)+ \frac{1}{\frac{p}{n}+z-1 +z\frac{p}{n} s(z)+o(\delta^2)}=0,$$ where we used Lemma \ref{lem:projection} to obtain that $\xi_{kk}=||X_k||^2/n=1+o(\delta^2)$ with overwhelming probability.

By assumption $p/n \rightarrow y$, when $n$ is large enough,
\begin{equation}\label{eq:s(z)}
s(z)+ \frac{1}{y+z-1 +yz s(z)+o(\delta^2)}=0
\end{equation}
holds with overwhelming probability.

In (\ref{eq:s(z)}), for the error term $o(\delta^2)$, one has either $\frac{o(\delta^2)}{y+z-1+yzs(z)}=o(\delta^2)$ or $y+z-1+yzs(z)=o(1)$. In the latter case, we get $s(z)=-\frac{y+z-1}{yz}+o(1)$. In the first case, we impose a Taylor expansion on (\ref{eq:s(z)}), 
$$s(z)(y+z-1+yzs(z))+1+o(\delta^2)=0.$$
Completing a perfect square for $s(z)$ in the above identity, one can solve the equation for $s(z)$, 
\begin{equation}\label{eq:eq1}
\sqrt{yz}(s(z)+\frac{y+z-1}{2yz})=\pm\sqrt{\frac{(y+z-1)^2}{4yz}-1+o(\delta^2)}
\end{equation}
If $\frac{o(\delta^2)}{\sqrt{\frac{(y+z-1)^2}{4yz}-1}}=o(\delta)$, by a Taylor expansion on the right hand side of (\ref{eq:eq1}), we have $\sqrt{yz}(s(z)+\frac{y+z-1}{2yz})=\pm \sqrt{\frac{(y+z-1)^2}{4yz}-1}+o(\delta)$. Therefore, $s(z)=s_{MP,y}(z)+o(\delta)$ or $s(z)=s_{MP,y}(z)-\frac{\sqrt{(y+z-1)^2 -4yz}}{yz}+o(\delta)=-s_{MP,y}(z)-\frac{y+z-1}{yz}+o(\delta)$. If $\frac{(y+z-1)^2}{4yz}-1=o(\delta^2)$, from (\ref{eq:eq1}) and the explicit formula for $s_{MP,y}(z)$, we still have $s(z)=s_{MP,y}(z)+o(\delta)$.

To summarize the above discussion, one has, with overwhelming probability,  either 
\begin{equation}\label{eq:diff1}
s(z)=s_{MP,y}(z)+o(\delta)
\end{equation}
or
\begin{equation}\label{eq:diff2}
s(z)=s_{MP,y}(z)-\frac{\sqrt{(y+z-1)^2 -4yz}}{yz}+o(\delta)=-s_{MP,y}(z)-\frac{y+z-1}{yz}+o(\delta)
\end{equation}
or
\begin{equation}\label{eq:diff3}
s(z)=-\frac{y+z-1}{yz}+o(1)
\end{equation}
We may assume the above trichotomy holds for all $z=x+\sqrt{-1}\eta$ with $a\le x\le b$ and $\eta_{0} \le \eta \le n^{10}/\delta$ where $\eta_{0}=\frac{K^2 \log^6 n}{n\delta^8}$.

When $\eta = n^{10}/\delta$, from $|s(z)| \le 1/\eta$ and $|s_{MP,y}(z)| \le 1/\eta$, we have $s(z)$ and $s_{MP,y}(z)$ are both $o(\delta)$ and therefore $(\ref{eq:diff1})$ holds in this case. By continuity, we conclude that either (\ref{eq:diff1}) holds in the domain of interest or there exists some $z$ in the domain such that (\ref{eq:diff1}) and (\ref{eq:diff2}) or (\ref{eq:diff1}) and (\ref{eq:diff3}) hold together.

On the other hand, (\ref{eq:diff1}) or (\ref{eq:diff3}) cannot hold at the same time. Otherwise, $s_{MP,y}(z)+\frac{y+z-1}{yz}=o(1)$. However, from $s_{MP,y}(z)(s_{MP,y}(z)+\frac{y+z-1}{yz}) = -\frac{1}{yz}$ and $|s_{MP,y}(z)| \le \frac{\sqrt{2}}{\sqrt{y}(1-\sqrt{y}+\sqrt{\eta_0})}$, one can see that $|s_{MP,y}(z)+\frac{y+z-1}{yz}|$ is bounded from below, which implies a contradiction.

Similarly, (\ref{eq:diff1}) or (\ref{eq:diff2}) cannot both hold except when $(y+z-1)^2-4yz=o(\delta^2)$. Otherwise, we can conclude that $2s_{MP,y}(z)+\frac{y+z-1}{yz}=o(\delta)$. From the explicit formula of $s_{MP,y}$, $$s_{MP,y}(z)+\frac{y+z-1}{yz}=\frac{\sqrt{(y+z-1)^2-4yz}}{yz}.$$One can conclude $|2s_{MP,y}(z)+\frac{y+z-1}{yz}| \ge C\delta$, which contradicts our assertion. Actually, if $(y+z-1)^2-4yz=o(\delta^2)$, (\ref{eq:diff1}) and (\ref{eq:diff2}) are equivalent.

In conclusion, (\ref{eq:diff1}) holds with overwhelming probability in the domain of interest. 

\section{Gap theorem and Four Moment theorem}\label{section:main}

In this section, we complete the proofs of the main results, Theorem \ref{thm:4main} and Theorem \ref{thm:gap-c1}. The proofs follow closely those in \cite{tvcovariance} (as well as in \cite{tvrandom}, \cite{tvrandom2}), so we shall focus on the changes needed to that argument. We assume substantial familiarity with the materials in \cite{tvrandom}, \cite{tvrandom2}, and will cite from them repeatedly. 

It is convenient to use the \emph{augmented matrix} 
\begin{equation}\label{eq:M}
\textbf{M}:=
\left(
\begin{array}{cc}
0 & M^*  \\
M & 0  \end{array}
\right)
\end{equation}
which is a $(p+n) \times (p+n) $ Hermitian matrix with eigenvalues $\pm \sigma_1(M),\ldots, \pm \sigma_p(M)$ and $n-p$ zeros. 
In this way, we can import the results obtained in \cite{tvrandom}, \cite{tvrandom2} and \cite{tvcovariance} to the model discussed in this paper. 

As mentioned in the beginning of Section \ref{section:main}, one can assume that 
$$|\zeta_{ij}|, |\zeta'_{ij}| \le n^{10/C_0}$$almost surely for all $i,j$. We also assume that the distributions of $M, M'$ are continuous to ensure the singular values are almost surely simple.

Let us first state a weaker version the Four Moment Theorem as we assume gap properties for the matrices considered:

\begin{thm}[Four Moment theorem with Gap assumption]\label{thm:4main-gap}
For sufficiently small $c_0 >0$ and sufficiently large $C_0 >0$ ($C_0=10^4$ will suffice) the following holds for every $k\ge 1$. Let $M=(\zeta_{ij})_{1\le i\le p, 1\le j\le n }$ and $M'=(\zeta'_{ij})_{1\le i\le p, 1\le j\le n }$ be two random matrices satisfying condition \textbf{C1} with the indicated constant $C_0$, and assume that for each $i,j$ that $\zeta_{ij}$ and $\zeta'_{ij}$ match to order 4. Let $W, W'$ be the associated covariance matrices. Assume also that $M, M'$ obey the gap property and $p/n\rightarrow y$ for some $0<y \le 1$.

Let $G:\mathbb{R}^k \rightarrow \mathbb{R}$ be a smooth function obeying the derivative bounds 
\begin{equation}\label{eq:derivative0}
|\nabla^j G(x)| \le n^{c_0}
\end{equation}
for all $0\le j \le 5$ and $x\in \mathbb{R}^k$.

Then for any $1\le i_1 < i_2<\ldots <i_k \le n$, and for $n$ sufficiently large depending on $k, c_0$, we have 
\begin{equation}\label{eq:conclusion0}
|\E(G(n\lambda_{i_1}(W),\ldots,n\lambda_{i_k}(W) ) ) - \E(G(n\lambda_{i_1}(W'),\ldots,n\lambda_{i_k}(W') ) )| \le n^{-c_0}.
\end{equation}

If $\zeta_{ij}$ and $\zeta'_{ij}$ only match to order 3 rather 4, then the conclusion (\ref{eq:conclusion0}) still holds provided that one strengthens (\ref{eq:derivative0}) to 
$$|\nabla^j G(x)| \le n^{-jc_1}$$ for all $0\le j \le 5$ and $x\in \mathbb{R}^k$ and any $c_1 >0$, provided that $c_0$ is sufficiently small depending on $c_1$.
\end{thm}

The Four Moment theorem follows directly from Theorem \ref{thm:gap-c1} and Theorem \ref{thm:4main-gap}. The next two sections are devoted to the proofs of Theorem \ref{thm:gap-c1} and Theorem \ref{thm:4main-gap}. 

\subsection{Proof of Theorem \ref{thm:4main-gap}}

The key technical step (also used in proving Theorem \ref{thm:gap-c1}) is the truncated Four Moment Theorem, which follows by applying \cite[Proposition 6.1 and Proposition 6.2]{tvrandom2} (or see \cite[Proposition 35]{tvcovariance}) to the \emph{argumented matrix} \textbf{M}. The proof is omitted here.

\begin{thm}[Truncated four moment theorem]\label{thm:4main-truncated}
For sufficiently small $c_0 >0$ and sufficiently large $C_0 >0$ ($C_0=10^4$ will suffice) the following holds for every $k\ge 1$. Let $M=(\zeta_{ij})_{1\le i\le p, 1\le j\le n }$ and $M'=(\zeta'_{ij})_{1\le i\le p, 1\le j\le n }$ be two random matrices satisfying condition \textbf{C1} with the indicated constant $C_0$, and assume that for each $i,j$ that $\zeta_{ij}$ and $\zeta'_{ij}$ match to order 4. Assume also that $|\zeta_{ij}|, |\zeta'_{ij}| \le n^{10/C_{0}}$ and $p/n\rightarrow y$ for some $0<y \le 1$.

Let $G:\mathbb{R}^k \times \mathbb{R}^k_{+} \rightarrow \mathbb{R}$ be a smooth function obeying the derivative bounds 
\begin{equation}\label{eq:derivative1}
|\nabla^j G(x_1,\ldots,x_k,q_1,\ldots,q_k)| \le n^{c_0}
\end{equation}
for all $0\le j \le 5$ and $x_1,\ldots, x_k \in \mathbb{R}$, $q_1,\ldots,q_k \in \mathbb{R}_{+}$, and such that $G$ is supported on the region $q_1,\ldots,q_k \le n^{c_0}$, and the gradient $\nabla$ is in all $2k$ variables.

Then for any $1\le i_1 < i_2<\ldots <i_k \le n$, and for $n$ sufficiently large depending on $\varepsilon, k, c_0$, we have 
\begin{equation}\label{eq:conclusion1}
\begin{split}
&|\E(G(\sqrt{n}\sigma_{i_1}(M),\ldots,\sqrt{n}\sigma_{i_k}(M), Q_{i_1}(M),\ldots, Q_{i_k}(M) ) ) \\
&- \E(G(\sqrt{n}\sigma_{i_1}(M'),\ldots,\sqrt{n}\sigma_{i_k}(M'), Q_{i_1}(M'),\ldots, Q_{i_k}(M') ) )| \le n^{-c_0}.
\end{split}
\end{equation}

If $\zeta_{ij}$ and $\zeta'_{ij}$ only match to order 3 rather 4, then the conclusion (\ref{eq:conclusion1}) still holds provided that one strengthens (\ref{eq:derivative1}) to 
$$|\nabla^j G(x_1,\ldots,x_k,q_1,\ldots,q_k)| \le n^{-jc_1}$$ for any $c_1 >0$, provided that $c_0$ is sufficiently small depending on $c_1$.

\end{thm}

As in the arguments in Section 6 in \cite{tvcovariance}, we use the qualities for $1\le i \le p$,
\begin{equation*}
\begin{split}
Q_i(\textbf{M})&:= \sum_{\lambda \neq \sigma_i(M)} \frac{1}{| \sqrt{n}(\lambda- \sigma_i(M)) |^2}\\
&=\frac{1}{n} (\sum_{1 \le j \le p: j\neq i} \frac{1}{| \sigma_j(M)-\sigma_i(M) |^2} +\frac{n-p}{\sigma_i(M)^2} +\sum_{j=1}^p \frac{1}{|\sigma_j(M)+\sigma_i(M)|^2} ).
\end{split}
\end{equation*}

The gap property (up to the edge) on M ensures an upper bound on $Q_i(\textbf{M})$. The proof repeats exactly the proof of Lemma 32 in \cite{tvcovariance}.
\begin{prop}\label{prop:Q}
If M satisfies the gap property, then for any $c_0 >0$(independent of n), and any $1\le i \le p$, one has $Q_i(\textbf{M}) \le n^{c_0}$ with high probability.
\end{prop} 

Now Theorem \ref{thm:4main-gap} follows by defining $\tilde{G}: \mathbb{R}^k \times \mathbb{R}_{+}^{k} \rightarrow \mathbb{R}$ to be $$\displaystyle\tilde{G}(\sqrt{n}\sigma_{i_1},\ldots,\sqrt{n}\sigma_{i_k}, Q_{i_1},\ldots, Q_{i_k}) := G(\sqrt{n}\sigma_{i_1},\ldots,\sqrt{n}\sigma_{i_k}) \prod_{j=1}^k\eta(Q_{i_j})$$
where $\eta(x)$ is a smooth cutoff to the region $x\le n^{c_0}$ which equals $1$ on $x\le n^{c_0}/2$. From Propositon \ref{prop:Q}, we have $$|\E(G(\sqrt{n}\sigma_{i_1}(M),\ldots,\sqrt{n}\sigma_{i_k}(M)))- \E(\tilde{G} (\sqrt{n}\sigma_{i_1}(M),\ldots,\sqrt{n}\sigma_{i_k}(M), Q_{i_1}(\textbf{M}),\ldots, Q_{i_k}(\textbf{M})))| \ll \frac{1}{n^c},$$
for some $c>0$, and a similar relation holds for $M'$. The proof is complete by using the above relations and Theorem \ref{thm:4main-truncated}.

\subsection{Proof of the Gap Theorem}\label{subsection:gap}

We first have a gap theorem under additional exponential decay hypothesis on the ensembles of $M$. The proof is presented in Section \ref{subsection:gap-expo}.

\begin{thm}[Gap theorem up to the edge]\label{thm:gap-expo} Let $M=(\zeta_{ij})_{1\le i\le p, 1\le j\le n }$ be a random matrix obeying condition \textbf{C1}, and the entries $\zeta_{ij}$ satisfy exponential decay in the sense that $\P(|\zeta_{ij}| \ge t^C) \le \exp(-t)$ for all $t\ge C'$ for all $i,j$ and some constants $C,C'>0$. Then $M$ obeys the gap property.
\end{thm}

The next observation is the following matching lemma (See Lemma 33 in \cite{tvcovariance}), which together with Theorem \ref{thm:4main-truncated}, ensures us to remove the exponential decay hypothesis in Theorem \ref{thm:gap-expo}.

\begin{lem}[Matching lemma,\cite{tvcovariance}]\label{matching}
Let $\zeta$ be a complex random variable with mean zero, unit variance, and third moment bounded by some constant $a$. Then there exists a complex random variable $\tilde \zeta$ with support bounded by the ball of radius $O_a(1)$ centered at the origin (and in particular, obeying the exponential decay hypothesis uniformly in $\zeta$ for fixed $a$) which matches $\zeta$ to third order.
\end{lem}

Now consider the matrix $M=(\zeta_{ij})_{1\le i\le p, 1\le j \le n}$ in Theorem \ref{thm:gap-c1}. By the matching lemma, we can find a random matrix $M'=(\zeta'_{ij})_{1\le i\le p, 1\le j \le n}$ such that $\zeta'_{ij}$ satisfies the exponential decay hypothesis and $\zeta'_{ij}$ matches $\zeta_{ij}$ to third order for each $i,j$. By Theorem \ref{thm:gap-expo}, the matrix $M'$ obeys the gap property. Similarly as in Section 6, \cite{tvcovariance}, let $\eta(x)$ be a smooth cutoff to the region $x\le n^{c_0}$. Then by Proposition \ref{prop:Q}, $\E\eta(Q_i(M')) = 1-O(n^{-c_1})$, which, by using Theorem \ref{thm:4main-truncated}, implies that $\E\eta(Q_i(M')) = 1-O(n^{-c_2})$  for some $c_2$ independent of $n$. Hence, $M$ also satisfies the gap property.

\subsection{Proof of Theorem \ref{thm:gap-expo}:}\label{subsection:gap-expo}

The proof follows closely to that discussed in Section 5, \cite{tvrandom2}. We shall mainly mention the corresponding changes. Interested readers can find the detailed proofs in \cite{tvrandom}. First, in order to operate an induction argument, we need to treat the edge case $i=1,p$ separately. 

\begin{prop}[Extreme cases]\label{prop:gap-edge} 
Theorem \ref{thm:gap-expo} is true when $i=1$ or $i=p$.
\end{prop}
\begin{proof}
By symmetry, it suffices to show for $i=p$. In the interlacing identity (Lemma \ref{lem:cauchyid}), 
$$\displaystyle \sum_{j=1}^{p-1} \frac{\frac{1}{n}\sigma_j(M_{p-1,n})^2 |u_j (M_{p-1,n})^* Y|^2}{\sigma_j(M_{p-1,n})^2 - \sigma_p(M_{p,n})^2} = \frac{||Y||^2}{n} - \frac{1}{n}\sigma_p(M_{p,n})^2.
$$
On the right-hand side of the above identity, $Y\in \mathbb{C}^n$, by Lemma \ref{lem:projection}, $||Y||^2/n = 1+o(1)$ with overwhelming probability. And $\sigma_p(M_{p,n})^2 /n =\lambda_p(W)=b+o(1)$ with overwhelming probability. All the terms in the left-hand side is are negative signs, thus 
$$\frac{\frac{1}{n}\sigma_j(M_{p-1,n})^2 |u_j (M_{p-1,n})^* Y|^2}{|\sigma_j(M_{p-1,n})^2 - \sigma_p(M_{p,n})^2|} \ll 1.$$
From Theorem \ref{thm:delocalization} and Lemma \ref{lem:tail}, one can conclude that $|u_{p-1}(M_{p-1,n})^* Y|^2 \ge n^{-c/10}$ with high probability. Therefore, $|\sigma_j(M_{p-1,n})^2 - \sigma_p(M_{p,n})^2| \ge n^{-c}$ with high probability. The conclusion follows from the Cauchy interlacing law.
\end{proof}

For the general case for the gap theorem, we write $i_0, p_0$ instead of $i,p$ and define $N_0:=p_0+n$, as in \cite{tvrandom}, we introduce the \emph{regularized gap}
\begin{equation}\label{eq:reggap}
\displaystyle g_{i,l,p}:= \text{inf}_{1 \le i_{-} \le i-l < i \le i_+ \le p}^{} \frac{ \sqrt{N_0} \sigma_{i_+}(M_{p,n}) - \sqrt{N_0} \sigma_{i_{-}}(M_{p,n}) } { \text{min} (i_{+} - i_{-}, \log^{C_1} N_0)^{\log^{0.9} N_0} },
\end{equation}
where $C_1 >1 $ is a large constant to be determined later. To show Theorem \ref{thm:gap-expo}, it is enough to show that $$g_{i_0,1,p_0} \le n^{-c_0},$$ for $1 < i_0 < p_0.$

By repeating the arguments in Section 3.5, \cite{tvrandom}, the proof relies on the following two key propositions. The idea is to propagate a narrow gap for $M_{p,n}$ backwards in $p$ until one can use Theorem \ref{thm:concentration} to control the occurrence of the gap.

\begin{prop}[Backwards propagation of gap].\label{prop:backwardsgap}
Suppose $p_0/2 \le p < p_0$ and $l \le p/10$ is such that 
\begin{equation}\label{eq:1}
g_{i_0,l,p+1} \le \delta
\end{equation}
for some $1< \delta \le 1$(which can depend on $n$), and that 
\begin{equation}\label{eq:edge}
\sqrt{n}\sigma_{p+1}(M_{p+1,n})-\sqrt{n}\sigma_{p}(M_{p+1,n}) \ge \delta \exp(\log^{0.91} n_0)
\end{equation}
Then $i_0 \le p$. Suppose further that
$$g_{i_0,l+1,p} \ge 2^m g_{i_0,l,p+1}$$ for some $m\ge 0$ with $$2^m \le \delta^{-1/2}.$$
Let $X_{p+1}$ be the $p+1^{\text{th}}$ row of $M_{p_0,n}$, and let $u_1(M_{p,n}), \ldots, u_p(M_{p,n})$ be an orthonormal system of right singular vectors of $M_{p,n}$ associated to $\sigma_1(M_{p,n}), \ldots, \sigma_p(M_{p,n})$. Then one of the following statement holds:
\begin{itemize}
\item[(i)](Macroscopic spectral concentration) There exists $1\le i_{-} < i_+ \le p+1$ with $i_+- i_- \ge \log^{C_1/2} n$ such that $| \sqrt{n}\sigma_{i_+}(M_{p+1,n})-\sqrt{n}\sigma_{i_-}(M_{p+1,n}) | \le \delta^{1/4} \exp(\log^{0.95}) (i_+-i_-).$
\item[(ii)](Small inner products) There exists $1\le i_- \le i_0 -l < i_0 \le i_+ \le p$ with $i_+ - i_- \le \log^{C_1/2} n$ such that $$\displaystyle\sum_{i_- \le j < i_+} |X^*_{p+1} u_j (M_{p,n})|^2 \le \frac{i_+ -i_-}{2^{m/2}\log^{0.01} n}.$$
\item[(iii)](Large singular value) For some $1\le i \le p+1$ one has $$| \sigma_i(M_{p+1,n}) | \ge \frac{\sqrt{n} \exp(-\log^{0.96}n)}{\delta^{1/2}}.$$
\item[(iv)](Large inner product) There exists $1\le i \le p$ such that $$|X^*_{p+1} u_i(M_{p,n})|^2 \ge \frac{\exp(-\log^{0.96}n)}{\delta^{1/2}}.$$
\item[(v)](Large row) We have $$|| X_{p+1}||^2 \ge \frac{n\exp(-\log^{0.96}n)}{\delta^{1/2}}.$$
\item[(vi)](Large inner product near $i_0$) There exists $1\le i \le p$ with $|i-i_0| \le \log^{C_1}n$ such that $$|X^*_{p+1} u_i(M_{p,n})|^2 \ge 2^{m/2} \log^{0.8}n.$$
\end{itemize}

\end{prop}

\begin{proof} Apply Lemma 5.3 in \cite{tvrandom2} to the \emph{augmented matrix} $$A_{p+n+1}:=\sqrt{n}
\left(
\begin{array}{cc}
0 & M^*_{p+1,n}  \\
M_{p+1,n} & 0  \end{array}
\right) $$ and $$A_{p+n}:=\sqrt{n}
\left(
\begin{array}{cc}
0 & M^*_{p,n}  \\
M_{p,n} & 0  \end{array}
\right). $$ Noticed $A_{p+n}$ is $A_{p+n+1}$ with the rightmost column and bottom column(which is $X_{p+1}$ and $p+1$ zeros) removed. The eigenvalues of $A_{p+n}$ are $\pm \sqrt{n}\sigma_1(M_{p,n}),\ldots,\pm \sqrt{n} \sigma_p(M_{p,n})$ and $0$, and an orthonormal eigenbasis includes the vectors 
$\left( \begin{array}{c} u_j(M_{p,n}) \\
v_j(M_{p,n}) \end{array} \right)$ for $1\le j \le p$. (The "Large coefficient" event in Lemma 5.3, \cite{tvrandom2}) cannot occur as $A_{p+n+1}$ has zero diagonals.)
\end{proof}

The next proposition claims that the events (i)-(vi) occurs with small probability.
\begin{prop}[Bad events are rare]. Suppose that $p_0/2 \le p < p_0$ and $l\le p/10$ and set $\delta:=n_0^{-\kappa}$ for some sufficiently small fixed $\kappa >0$. Then
\begin{itemize}
\item[(a)]The events (i), (iii), (iv), (v) in Proposition \ref{prop:backwardsgap} all fail with high probability.
\item[(b)]There is a constant $C'$ such that all the coefficients of the right singular vectors $u_j(M_{p,n})$ for $1\le j \le p$ are of magnitude at most $n^{-1/2}\log^{C'}n$ with overwhelming probability. Conditioning $M_{p,n}$ to be a matrix with this property, the events (ii) and (vi) occur with a conditional probability of at most $2^{-\kappa m} +n^{-\kappa}$.
\item[(c)]Furthermore, there is a constant $C_2$ (depending on $C', \kappa, C_1$) such that if $l\ge C_2$ and $M_{p,n}$ is conditioned as in (b), then (ii) and (vi) in fact occur with a conditional probability of at most $2^{-\kappa m}\log^{-2C_1} n + n^{-\kappa}$.
\end{itemize}
\end{prop}

The proof of the above proposition repeats the proof of Proposition 53 in \cite{tvrandom} with the major difference being that Theorem \ref{thm:concentration} and Theorem \ref{thm:delocalization} are applied instead of Theorem 60 and Proposition 62 in \cite{tvrandom}.


\bibliographystyle{plain}
\bibliography{covariance-edge}


\end{document}